\documentclass[11pt]{amsart}

\usepackage{amsfonts,amssymb,enumerate,color,bm,array}
\usepackage{mathrsfs}
\usepackage[colorlinks=true,citecolor=blue,urlcolor=blue,linkcolor=blue,pagebackref]{hyperref}
\usepackage{tikz-cd}
\usepackage[shortlabels]{enumitem}
\usepackage{mathtools}
\usepackage[noabbrev,capitalise]{cleveref}

\hypersetup{
    pdftitle={The tangent bundle of a generalized Kummer fourfold is obstructed},
    pdfauthor={Alessio Bottini}
}

\def\C{\ensuremath{\mathbb{C}}}

\def\cO{\ensuremath{\mathcal O}}

\DeclareMathOperator{\End}{End}
\DeclareMathOperator{\Ext}{Ext}

\DeclareMathOperator{\Id}{Id}
\DeclareMathOperator{\Ker}{Ker}

\DeclareMathOperator{\Sym}{Sym}

\def\isom{\simeq}

\newcommand{\wt}{\widetilde}

\newcommand{\mor}[1][]{\xrightarrow{#1}}
\newcommand{\isomor}{\mor[\sim]}
\def\hra{\hookrightarrow}

\newtheorem{Thm}{Theorem}[section]

\newtheorem{Lem}[Thm]{Lemma}
\newtheorem{Cor}[Thm]{Corollary}

\theoremstyle{definition}

\newtheorem{Rem}[Thm]{Remark}

\allowdisplaybreaks
\setlist[itemize]{noitemsep,nolistsep}
\setlist[enumerate]{noitemsep,nolistsep}

\begin{document}

\title{The tangent bundle of a generalized Kummer fourfold is obstructed}

\author[Alessio Bottini]{Alessio Bottini}

\address{Mathematisches Institut, Universit\"at Bonn, Endenicher Allee 60, 53115 Bonn, Germany}
\email{bottini@math.uni-bonn.de}

\makeatletter
\@namedef{subjclassname@2020}{\textup{2020} Mathematics Subject Classification}
\makeatother

\begin{abstract}
For a hyper-K\"ahler fourfold $X$ of generalized Kummer type, we construct an infinitesimal deformation of the tangent bundle $T_X$ with non-zero Yoneda square.
This is the first example of a stable hyperholomorphic bundle with obstructed deformations.
\end{abstract}

\maketitle

\section{Introduction}

Stable hyperholomorphic bundles are among the most natural sheaves on a hyper-K\"ahler manifold.
Their Hermitian--Yang--Mills connections remain holomorphic with respect to every complex structure induced by the hyper-K\"ahler metric, and their deformation theory is governed by a quadratic Kuranishi map.
In fact, Verbitsky proved in \cite[Theorem 6.2]{verbitsky96} that the Kuranishi map agrees with the Yoneda pairing
\[
    \cup \colon \Ext^1(F,F)\times \Ext^1(F,F)\longrightarrow \Ext^2(F,F).
\]
More recently, Meazzini and Onorati deduced this quadraticity result from the stronger property that the DGLA controlling their deformation is formal \cite{meazzini_onorati23}.
This raises the basic question of whether the moduli space of a stable hyperholomorphic bundle can actually be singular.
In this paper, we show that this happens for the tangent bundle on a hyper-K\"ahler fourfold of generalized Kummer type. 

\begin{Thm}
\label{thm:introduction}    
    Let $X$ be a hyper-K\"ahler manifold of $\mathrm{Kum}_2$-type.
    There is a class $\gamma \in \Ext^1(T_X,T_X)$ such that
    \[
        \gamma\cup \gamma\ne 0 \in \Ext^2(T_X,T_X).
    \]
    In particular, the Kuranishi space $\mathrm{Def}(T_X)$ is singular at $[T_X]$.
\end{Thm}

This question fits into the broader study of sheaves on hyper-K\"ahler manifolds, an area that has received considerable attention recently; see, for example, \cite{ogrady22,ogrady_kummer24,ogrady_rigid24,ogrady_many_moduli26,ogrady_projective_bundles26,markman24,beckmann23,beckmann25,guoliu25,bottini24,bottini_tenfolds24,maulik_shen_yin_zhang25}.
For recent surveys, we refer to \cite{ogrady_survey26,bottini_macri_stellari26}.
While these developments have brought to the construction of several examples of stable (projectively) hyperholomorphic bundles, all of these examples either have a smooth deformation space, or it is not known whether this is the case. 
In this sense, \cref{thm:introduction} produces the first example of a stable hyperholomorphic bundle with obstructed deformations.  

We briefly describe the construction of the deformation. 
First of all, by results of Verbitsky, the cohomology of hyperholomorphic bundles does not change along twistor lines. 
Therefore, we may specialize $X$ to a chosen variety, and, for later purposes, we choose $X$ to be of the form $K_2(A)$ with $A = E \times F$.

The construction will provide a linear injection
\begin{equation}\label{eq:mainInjection}
    H^1(A,\cO_A)
    \hookrightarrow H^1(X,\Omega^2_X)
    \hookrightarrow H^1(X,\End(T_X)),
    \qquad
    u\longmapsto \gamma_u,
\end{equation}
which exists for any abelian surface $A$; the product assumption will be used to check that $\gamma_u \cup \gamma_u \neq 0.$

The construction itself is quite simple. 
Denoting $P \subset A^3$ the kernel of the summation map, we have a diagram 
\[\begin{tikzcd}
	X && P \\
	& Y
	\arrow["\pi"', from=1-1, to=2-2]
	\arrow["q", from=1-3, to=2-2]
\end{tikzcd}\]
where $q$ is the quotient by $\mathfrak{S}_3$ and $\pi$ is the Hilbert--Chow map. 
We show in \cref{lem:lift} that there is an inclusion $H^1(P,\Omega_P^2)^{\mathfrak{S}_3} \subset H^1(X,\Omega_X^2)$. 
Concretely, the map is given by first restricting to the open $P^{\circ}$ where the $\mathfrak{S}_3$-action is free, then descending to the image $Y^{\circ} = q(P^{\circ})$, and then extending the form to $X$. 
One computes that 
\begin{equation}\label{eq:InvariantPart}
H^1(P,\Omega_P^2)^{\mathfrak{S}_3} = H^1(A,\cO_A) \otimes H^0(A,\Omega_A^2).
\end{equation}
Hence, after choosing a $2$-form on $A$, we get the desired inclusion \eqref{eq:mainInjection}.

It remains to see that the square of the class in the images is non-zero.
The difficulty here consists in giving concrete descriptions for elements in $\Ext^2(T_X,T_X)$. 
To do this, we restrict the square to the open $X^{\circ} = \pi^{-1}(Y^{\circ})$, where we can pull-back to the more familiar $P^{\circ}$: in \cref{lem:square-free} we compute it explicitly.
In fact it will have the form 
\[
    \lambda \otimes B \in \Ext^2(T_{P^{\circ}},T_{P^{\circ}}) = H^2(P^{\circ},\cO_{P^{\circ}}) \otimes \End(H^0(P^{\circ},T_{P^{\circ}})).
\]
The non-vanishing of $B$ will follow from the representation theory of $\mathfrak{S}_3$ while that of $\lambda$ will require the product assumption $A = E \times F$.
This assumption will allow to find a complete surface $S \subset P^{\circ}$, such that $\lambda|_S$ does not vanish.

The behavior of the tangent bundle is quite different for other deformation types.
For instance, the tangent bundle of a manifold of $\mathrm{K3}^{[2]}$-type is rigid by \cite{gavran21}, and it is expected to be rigid for $\mathrm{K3}^{[n]}$-type for all $n$.

\noindent\textbf{AI disclosure.} The main result was obtained with the assistance of OpenAI's GPT-5.6 Sol. The author wrote the paper and takes responsibility for its contents.

\section{Basics on the generalized Kummer}
\label{sec:kummer}

Let $A$ be an abelian surface and consider the summation morphisms
\[
    s:A^{[3]}\longrightarrow A,
    \qquad
    s_0:A^3\longrightarrow A.
\]
The generalized Kummer fourfold is the fiber
\[
    X=K_2(A) \coloneqq s^{-1}(0).
\]
It is a projective irreducible holomorphic symplectic fourfold whenever $A$ is projective, see \cite[Section 7]{beauville83}.
Set
\[
    P=\Ker(s_0),\qquad Y=P/\mathfrak S_3,
\]
where $\mathfrak S_3$ acts by permuting the factors.
The Hilbert--Chow morphism restricts to a symplectic resolution
\[
    \pi:X\longrightarrow Y.
\]

Let $P^\circ\subset P$ be the locus of triples with pairwise distinct entries and put $Y^\circ=P^\circ/\mathfrak S_3$.
The action on $P^\circ$ is free, $Y^\circ$ is the smooth locus of $Y$, and $\pi$ is an isomorphism over $Y^\circ$. We will identify $X^\circ:=\pi^{-1}(Y^\circ)$ with $Y^\circ$ tacitly, and denote by 
\[
    q:P\longrightarrow Y,\qquad q^\circ:P^\circ\longrightarrow X^\circ
\]
the quotient morphisms.

The following Lemma will be the main mechanism to transfer information from $P$, which is easier to control, to $X$.

\begin{Lem}\label{lem:lift}
    For every $p\geq 0$, there are natural isomorphisms
    \[
        \Omega_Y^{[p]}\isomor (q_*\Omega_P^p)^{\mathfrak S_3}
        \qquad\text{and}\qquad
        \pi_*\Omega_X^p\isomor \Omega_Y^{[p]}.
    \]
    In particular, there is a natural injection
    \[
        H^1(P,\Omega_P^2)^{\mathfrak S_3}
        \hra H^1(X,\Omega_X^2).
    \]
    The restriction of the image to $X^\circ$ is the descent of the original class along $q^\circ$.
\end{Lem}

\begin{proof}
    The variety $Y$ has quotient singularities, hence rational singularities.
    By \cite[Corollary 1.8]{kebekus21}, every reflexive differential form on $Y$ extends uniquely to $X$, i.e.\
    \[
        \pi_*\Omega_X^p\isom \Omega_Y^{[p]}.
    \]
    Since $q$ is finite, we have $R^iq_*=0$ for $i>0$.
    On $P^\circ$, the action of $\mathfrak S_3$ is free.
    Therefore, $(q_*\Omega_P^p)^{\mathfrak S_3}$ and $\Omega_Y^{[p]}$ agree over $Y^{\circ}$.
    Since both sheaves are reflexive and $Y\setminus Y^{\circ}$ has codimension $2$, this isomorphism extends to $Y$.
    In particular
    \[
        H^1(Y,\Omega_Y^{[2]})
        \isom H^1(P,\Omega_P^2)^{\mathfrak S_3}.
    \]
    Finally, the five-term exact sequence for the Leray spectral sequence of $\pi$ starts with
    \[
        0\longrightarrow H^1(Y,\pi_*\Omega_X^2)
        \longrightarrow H^1(X,\Omega_X^2).
    \]
    This proves the injection.
    Every map in the construction is compatible with restriction to $Y^\circ$, where $\pi$ is an isomorphism, which gives the last assertion.
\end{proof}

\begin{Rem}
A nonzero holomorphic two-form $\omega$ on $A$ induces an $\mathfrak S_3$-invariant symplectic form $\sigma_P$ on $P$. Over $P^\circ$, it descends to $Y^\circ\simeq X^\circ$ and extends to the symplectic form $\sigma_X$ on $X$. We choose the forms compatibly, so that $(q^\circ)^*(\sigma_X|_{X^\circ})=\sigma_P|_{P^\circ}$.

\end{Rem}

\section{Construction of the infinitesimal deformation}
\label{sec:class}
In this section we describe the infinitesimal deformation of $T_X$ that will have non-zero square. 
Set
\[
    W=H^0(A,T_A),\qquad U=H^1(A,\cO_A),
\]
and let
\[
    \rho=\bigl\{(z_1,z_2,z_3)\in \C^3\mid z_1+z_2+z_3=0\bigr\}
\]
be the standard representation of $\mathfrak S_3$.
There are equivariant identifications
\begin{equation}
\label{eq:tangent-P}
    T_P\isom \cO_P\otimes(W\otimes\rho),
    \qquad
    H^1(P,\cO_P)\isom U\otimes\rho.
\end{equation}
If $q_{\rho}$ denotes the restriction of the standard Euclidean pairing on $\C^3$, then the symplectic form $\sigma_P$ is the tensor product of the symplectic form $\omega$ on $W$ and $q_\rho$.

There is a splitting in irreducible representations
\[
\Sym^2 \rho = \C q_{\rho}^{-1} \oplus \rho. 
\]
In particular, there is a unique, up to scaling, $\mathfrak S_3$ equivariant morphism
\[
\star \colon \Sym^2 \rho \longrightarrow \rho.
\]
In coordinates it is given by
\[
(a_1,a_2,a_3) \star (b_1,b_2,b_3) = (a_1b_1,a_2b_2,a_3b_3)
-\frac{q_\rho(a,b)}3(1,1,1).
\]
For any $a \in \rho$ define $A_a \coloneqq a \star - \in \End(\rho)$.
For instance, for the orthonormal basis
    \[
        e_1=\frac1{\sqrt2}(1,-1,0),
        \qquad
        e_2=\frac1{\sqrt6}(1,1,-2),
    \]
one obtains
    \[
        A_{e_1}=\frac1{\sqrt6}
        \begin{pmatrix}0&1\\1&0\end{pmatrix},
        \qquad
        A_{e_2}=\frac1{\sqrt6}
        \begin{pmatrix}1&0\\0&-1\end{pmatrix}.
    \]
The assignment $a \mapsto A_a$ identifies $\rho$ with the self adjoint endomorphisms of $\rho$ of trace zero.

\begin{Rem}
\label{rem:matrices}
    Let $e_1,e_2$ be a $q_\rho$-orthonormal basis of $\rho$.
    Then $A_{e_1}$ and $A_{e_2}$ do not commute.
    Indeed, they are linearly independent because $a\mapsto A_a$ is an isomorphism, whereas two non-zero commuting trace-free endomorphisms of a two-dimensional vector space are proportional.
\end{Rem}

From \eqref{eq:tangent-P}, we write
\[
    H^1(P,\Omega_P^2) = H^1(P,\cO_P) \otimes H^0(P,\Omega^2_P) 
    =U\otimes\rho\otimes\bigwedge^2(W\otimes\rho)^\vee.
\]
For $u\in U$, the self-adjointness of the $A_b$ shows that the formula
\[
    \rho\otimes\rho\longrightarrow H^1(P,\Omega_P^2),
    \qquad
    a\otimes b\longmapsto
    (u\otimes a)\otimes
    \sigma_P\bigl((\Id_W\otimes A_b)\,\cdot,\cdot\bigr)
\]
is well defined, and it defines an $\mathfrak S_3$-equivariant map.
We define $\alpha_u$ to be the image under this map of the coevaluation tensor $q_\rho^{-1}\in\rho\otimes\rho$.
Namely, writing $q_{\rho}^{-1} = e_1 \otimes e_1 + e_2 \otimes e_2$, 
\begin{equation}\label{eq:ClassExplicit}
    \alpha_u=
    \sum_{i=1}^2
    (u\otimes e_i)\otimes
    \sigma_P\bigl((\Id_W\otimes A_{e_i})\,\cdot,\cdot\bigr).
\end{equation}
Since $q_\rho^{-1}$ is $\mathfrak S_3$-invariant, so is $\alpha_u$.
Let $\wt\alpha_u\in H^1(X,\Omega_X^2)$ be the image of $\alpha_u$ under the injection of \cref{lem:lift}.

The symplectic form $\sigma_X$ gives an identification $\Omega_X \simeq T_X$, hence a split injection
\begin{equation}\label{eq:Injection}
    H^1(X,\Omega_X^2) \hookrightarrow H^1(X,\End(T_X)) = \Ext^1(T_X,T_X).
\end{equation}
We define $\gamma_u$ to be the image of $\wt\alpha_u$ under \eqref{eq:Injection}.

\begin{Rem}
    Since
    \[
        \bigwedge^2(W\otimes\rho)^\vee
        \isom
        \bigl(\bigwedge^2W^\vee\otimes\Sym^2\rho\bigr)
        \oplus
        \bigl(\Sym^2W^\vee\otimes\bigwedge^2\rho\bigr),
    \]
    while $\Sym^2\rho\isom\C\oplus\rho$ and $\bigwedge^2\rho$ is the sign representation, we have
    \[
        H^1(P,\Omega_P^2)^{\mathfrak S_3}
        \isom U\otimes\bigwedge^2W^\vee,
    \]
    as claimed in \eqref{eq:InvariantPart}.
    Under this identification, $u\mapsto\alpha_u$ is, up to a nonzero scalar, $u\mapsto u\otimes\omega$.
\end{Rem}

\section{The Yoneda square}
\label{sec:square}

We first compute the square of $\gamma_u$ over the free locus.
For $i=1,2$, set
\[
    \lambda_i=(u\otimes e_i)|_{P^{\circ}}\in H^1(P^{\circ},\cO_{P^{\circ}}),
    \qquad
    A_i=A_{e_i}\in\End(\rho),
    \qquad
    B_i=\Id_W\otimes A_i\in\End(W\otimes\rho).
\]

\begin{Lem}
\label{lem:square-free}
    After restricting $\gamma_u$ to $X^\circ$ and pulling back along $q^\circ$, one has
    \begin{equation}
    \label{eq:square-free}
        (q^\circ)^*\bigl((\gamma_u\cup \gamma_u)|_{X^\circ}\bigr)
        =\lambda_1\cup\lambda_2
        \otimes [B_1,B_2] \in \Ext^2(T_{P^{\circ}},T_{P^{\circ}}),
    \end{equation}
    where we identify
    \[
        \Ext^2(T_{P^\circ},T_{P^\circ})
        =H^2(P^\circ,\End(T_{P^\circ}))
        =H^2(P^\circ,\cO_{P^\circ})\otimes\End(W\otimes\rho).
    \]

\end{Lem}

\begin{proof}
    The compatibility of $\sigma_P$ and $\sigma_X$ gives a commutative diagram
    \[
    \begin{tikzcd}[column sep=large,row sep=large]
    H^1(X^\circ,\Omega_{X^\circ}^2)
        \arrow[r,hook]
        \arrow[d,"(q^\circ)^*"']
    &
    H^1(X^\circ,\End(T_{X^\circ}))
        \arrow[d,"(q^\circ)^*"]
    \\
    H^1(P^\circ,\Omega_{P^\circ}^2)
        \arrow[r,hook]
    &
    H^1(P^\circ,\End(T_{P^\circ})).
    \end{tikzcd}
    \]
    Here the horizontal arrows are induced by the respective symplectic forms.
    By the last assertion of \cref{lem:lift}, the left vertical arrow sends
    $\wt\alpha_u|_{X^\circ}$ to $\alpha_u|_{P^\circ}$.
    Expanding $q_\rho^{-1}=e_1\otimes e_1+e_2\otimes e_2$, the lower horizontal arrow sends $\alpha_u|_{P^\circ}$ to $\lambda_1\otimes B_1+\lambda_2\otimes B_2$.
    Since the upper horizontal arrow sends $\wt\alpha_u|_{X^\circ}$ to $\gamma_u|_{X^\circ}$, commutativity gives
    \[
        (q^\circ)^*(\gamma_u|_{X^\circ})
        =\lambda_1\otimes B_1+\lambda_2\otimes B_2
        \quad\text{in}\quad
        H^1(P^\circ,\End(T_{P^\circ})).
    \]
    Since the $\lambda_i$ have cohomological degree one, we obtain
    \begin{align*}
        (\lambda_1\otimes B_1+\lambda_2\otimes B_2)^2
        &=(\lambda_1\cup\lambda_2)\otimes B_1B_2
          +(\lambda_2\cup\lambda_1)\otimes B_2B_1\\
        &=(\lambda_1\cup\lambda_2)\otimes[B_1,B_2],
    \end{align*}
    which is \cref{eq:square-free}.
\end{proof}

Notice that both $\lambda_1\cup\lambda_2$ and $[A_1,A_2]$ transform according to the sign representation of $\mathfrak S_3$.
Their tensor product is therefore invariant, as it must be.
Moreover, $[A_1,A_2]$ is trace-free, so the obstruction in \cref{eq:square-free} has zero trace.

It remains to prove that the class in \eqref{eq:square-free} does not vanish. By \cref{rem:matrices}, the commutator $[A_1,A_2]$ is non-zero, but a priori it could be that $\lambda_1 \cup \lambda_2$ vanishes in $H^2(P^{\circ},\cO_{P^{\circ}})$. 
To show that this is not the case, now specialize to a product abelian surface $A = E \times F$.
There is a product decomposition
\[
    P=P_E\times P_F,
    \qquad
    P_E \coloneqq \Ker(E^3\mor[+]E),
    \quad
    P_F \coloneqq \Ker(F^3\mor[+]F).
\]

\begin{Lem}
\label{lem:nonzero-square}
    Let $0\ne u\in H^1(E,\cO_E)\subset H^1(A,\cO_A)$.
    Then $ \lambda_1 \cup \lambda_2 \ne0 \in H^2(P^{\circ},\cO_{P^{\circ}})$. 
\end{Lem}

\begin{proof}
    Choose a point $ b=(b_1,b_2,b_3)\in P_F $ whose entries are pairwise distinct, and consider the complete abelian surface
    \[
        S \coloneqq P_E\times\{b\}\subset P.
    \]
    Every triple parametrized by $S$ has pairwise distinct entries in $A$, hence $S\subset P^\circ$.

    Under the identification
    \[
        H^1(S,\cO_{S})\isom H^1(E,\cO_E)\otimes\rho,
    \]
    the restrictions of $\lambda_1=u\otimes e_1$ and $\lambda_2=u\otimes e_2$ form a basis of $H^1(S,\cO_{S})$.
    Therefore,
    \[
        (\lambda_1\cup\lambda_2)|_S\ne0\ 
         \in H^2(S,\cO_S) = \wedge^2H^1(S,\cO_S).
    \]
\end{proof}

\begin{Cor}
\label{cor:plane}
    Let $A = E \times F$ be a product abelian surface. 
    Then, the construction $u\mapsto \gamma_u$ gives an injective linear map
    \[
        H^1(A,\cO_A)\hra \Ext^1(T_X,T_X)
    \]
    and every non-zero class in its image has non-zero Yoneda square.
\end{Cor}

\begin{proof}
    To show that $\gamma_u \cup \gamma_u \neq 0$ we can restrict to $X^{\circ}$ and pullback via $q^{\circ}$, and show that the result is not zero. 
    By \cref{lem:square-free}, we need to show that 
    \[
    \lambda_1\cup\lambda_2
        \otimes \Id_W\otimes[A_1,A_2] \neq 0.
    \]

    Write $u=u_E+u_F$ according to the decomposition
    \[
        H^1(A,\cO_A)=H^1(E,\cO_E)\oplus H^1(F,\cO_F). 
    \]
    If $u \neq 0$, without losing generality we may assume that $u_E\ne0$.
    Then its restriction to $S=P_E\times\{b\}$ kills $u_F$ and sends $\lambda_i$ to $u_E\otimes e_i$.
    Hence by \cref{lem:nonzero-square} we see that $\lambda_1 \cup \lambda_2 \neq 0 \in H^2(P^{\circ},\cO_{P^{\circ}})$. 
    On the other hand, the commutator $[A_1,A_2]$ does not vanish by \cref{rem:matrices}, which gives the result.
\end{proof}

We can now prove \cref{thm:introduction} from the introduction. 

\begin{proof}[Proof of \cref{thm:introduction}]
Let $X$ be of $\mathrm{Kum}_2$-type. 
By \cite[Theorem~3.1]{verbitsky_cohom96}, we can connect $X$ through a path of twistor lines with $K_2(A)$, with $A = E \times F$. 
Moreover, by \cite[Corollary~8.2]{verbitsky96} the $\Ext^*$ algebra of a hyperholomorphic bundle stays constant along each twistor line. 
The tangent bundle of any hyper-K\"ahler manifold is hyperholomorphic, because $X$ admits Ricci flat metrics by Yau's theorem \cite{yau78}.
In particular, there is a non-canonical isomorphism of graded algebras 
\[
    \Ext^*(T_X,T_X) \cong \Ext^*(T_{K_2(A)},T_{K_2(A)}).
\]
If $\gamma$ is the image of any non-zero class given by \cref{cor:plane}, its square $\gamma \cup \gamma$ is not zero. 
Consequently, $\mathrm{Def}(T_X)$ is not smooth at $[T_X]$.
Since the tangent bundle is stable by \cite{uhlenbeck86}, the same is true for the moduli space $M_{v(T_X)}(X,h)$ parametrizing $h$-polystable deformations of $T_X$, for any polarization $h$. 
\end{proof}

\end{document}